\documentclass{article}

\usepackage{ja_eng}
\usepackage[dvips]{graphicx}
\usepackage[ansinew]{inputenc}
\usepackage{xcolor}
\usepackage{amsfonts}
\usepackage{amsmath}
\usepackage{amssymb}
\usepackage{caption}
\usepackage{graphicx}
\usepackage{amsfonts}
\usepackage{amsthm}
\newtheorem{theorem}{Theorem}[section]
\newtheorem{lemma}[theorem]{Lemma}
\newcommand{\QED}{\hfill $\blacksquare$}
\usepackage{algorithm}
\usepackage{algpseudocode}
\usepackage{xcolor}

\def\set#1#2{\{ \; #1 \;:\;#2\;\}} 

\def\Sum#1#2{\sum\limits_{#1}^{#2}}

\newcommand {\bsis} {\left\{ \begin{array} }
\newcommand {\esis} {\end{array}\right.}
\def\conv#1#2{ \left( \begin{array}{c} #1 \\ #2 \\ \end{array}\right)} 
\def\bmat#1{\left[\begin{array}{#1}}
\def\emat{\end{array}\right]}
\def\bv {\bmat{c} }
\def\ev{\emat} 
\def\T{^T}
\def\fracg#1#2{{\displaystyle{\frac{#1}{#2}}}}

\def\Id{{\rm{I}}}
\def\hy{\hat{y}}
\def\rmax#1{{\rm{max}}^{(#1)}}
\def\tx{\tilde{x}}
\def\ty{\tilde{y}}

\newcommand{\blista}{\renewcommand{\labelenumi}{(\roman{enumi})} 
	\begin{enumerate}}
	\newcommand{\elista}{\end{enumerate} \renewcommand{\labelenumi}{\arabic{enumi}.}}

\newcommand{\W}{\mathbb{W}}

\newcommand{\R}{\mathbb{R}}
\newcommand{\D}{\mathbb{D}}

\newcommand{\E}{\mathbb{E}}

\newcommand{\F}{{\cal{F}}}

\def\vv#1{{\rm{#1}}}
\def\trace{{\rm{tr}}\,}
\newcommand{\diag}{{\rm diag}}

\def\NEW#1{\textcolor{black}{#1}}

\newtheorem{definition}{Definition}
\newtheorem{property}{Property}
\newtheorem{remark}{Remark}
\newtheorem{corollary}{Corollary}

\newcommand{\Pw}{{\rm{Pr}_\W}}
\newcommand{\PwN}{{\rm{Pr}_{\W^N}}}

\def\NATS#1{[#1]}
\newcommand{\Bin}{\mathbf{B}}

\title{Prediction error quantification through probabilistic scaling}

\author{Victor Mirasierra\thanks{\,\,\,\,Departamento de Ingenier\'ia de Sistemas y Autom\'atica, Universidad de Sevilla, Escuela Superior de Ingenieros, Sevilla, Spain.},\, Martina Mammarella\thanks{\,\,\,\,CNR-IEIIT; c/o Politecnico di Torino; Corso Duca degli Abruzzi 24, Torino; Italy.},\, Fabrizio Dabbene\footnotemark[2]\, \and Teodoro Alamo\footnotemark[1]}

\begin{document}

\maketitle

\begin{abstract}
In this paper, we address the probabilistic error quantification of a general class of prediction methods. We consider a given prediction model and show how to obtain, through a sample-based approach, a probabilistic upper bound on the absolute value of the prediction error. The proposed scheme is based on a probabilistic scaling methodology in which the number of required randomized samples is independent of the complexity of the prediction model. The methodology is extended to address the case in which the probabilistic uncertain quantification is required to be valid for every member of a finite family of predictors. We illustrate the results of the paper by means of a numerical example.
\end{abstract}


\section{Introduction and Problem Formulation}

Quantifying the error related to the process of approximating a set of given data with a prescribed prediction method represents a fundamental \NEW{requirement}, which has given rise to an entire research area known as \textit{Uncertainty Quantification}, see e.g.~\cite{mcclarren2018:Uncertainty:Quantification,VILLAVERDE201945} and references therein. 

\NEW{Motivated by this necessity, methods for directly constructing predictive models with prescribed robustness guarantees have recently gained popularity. For instance,
\cite{pierreinterval} presents several methods based on interval analysis to construct intervals which are guaranteed to contain the true value, under the assumption of deterministically bounded noise.} \NEW{Similarly,}
 data-based approaches exploiting the availability of random samples, \NEW{providing probabilistic guarantees, are being developed.}  These methods extend classical quantile regression \cite{Quantile:Regression}. In particular, we point out the probabilistic interval predictions proposed in~\cite{campi2009interval,garatti2019class}. 

\NEW{All these methods require to design (or re-design) the estimator using a specific 
ad-hoc model. However, this approach may not result practical when data-analysts have already constructed a model exploiting a ``preferred" technique (e.g., one based on support vector machines (SVM)) and they want to assess, before deployment, the actual uncertainty of their model.}

\NEW{For this reason, researchers have started investigating post-processing methods for quantifying the uncertainty of a \textit{given} predictor. This means that no new methodology is proposed for constructing a regression model but that such a model (or a family of candidate ones) is given. This philosophy is exactly the one pursued in uncertainty quantification methods, see e.g.\ the recent approaches based on polynomial chaos \cite{mcclarren2018:Uncertainty:Quantification}, or the \textit{conformal predictors} \cite{balasubramanian2014conformal}. These methods typically use additional \textit{validation} (or \textit{calibration})  data to determine precise levels of confidence in new predictions \cite{shafer2008tutorial}.}

In this paper, we move a step further in this direction and present sampling-based techniques for assessing the corresponding error in a \textit{computationally} efficient way. Indeed, this novel approach, extending recent results on probabilistic scaling, e.g.,
\cite{Alamo:18,Mammarella:CCTA:2020}, requires a number of randomized samples independent of the complexity of the prediction model \NEW{(i.e.\ the dimension of the regressor)}. 


In particular, we consider that given $x\in \R^{n_x}$, an estimation $\hy$ for $y\in \R$ is provided by operator $T:\R^{n_x}\to \R$. 
That is, $$ \hy=T(x).$$ 
We assume that the operator $T$ is a given predictive model that has been designed by means of any modelling methodology (first principles, linear regression, SVM regression, neural network, etc.). 

We want to provide a probabilistic bound on the prediction error. 
More formally, we consider the \NEW{random vector} $w=(x,y)\in \R^{n_x}\times \R\subseteq \W$, with \NEW{stationary} probability distribution $\Pw$, and we aim at constructing a function $\rho:\R^{n_x}\to\R$ such that, with probability no smaller than~$(1-\delta)$,  
$$\Pw\{ |y-T(x)| \leq \rho(x)\}\geq 1-\varepsilon.$$
The method relies on the possibility of accessing random data, that is, 
observations couples $w=(x,y)$. \NEW{Note that these are new data not used to construct $T(\cdot)$.}
We show that the sample complexity of the proposed techniques (i.e. the number of observations required) does
not depend on the chosen regression model but only on the desired probabilistic levels.

The remainder of the paper is structured as follows. In Section~\ref{sec:prob:up:bound}
we propose a first \NEW{simple} result, which allows to obtain an initial probabilistic bound on the prediction error via probabilistic maximization, given a predictive model. The obtained bound, which can be computed by means of a simple algorithm, is independent of the given $x$.  Section~\ref{sec:uncertainty:quantification}, focuses on including those situations in which the expected size of the error does depend on $x$, and propose a probabilistic bound conditioned to $x$. This approach is extended in Section~\ref{sec:finite:families}  to the case when a ``family" of candidates estimators is considered. Finally, in Section~\ref{sec:kernel:prediction}, we included the possibility of designing the predictors by means of kernel methods, providing at the same time also a measure of the expected size of the error. All these approaches are illustrated by means of a running numerical example, and conclusions are drawn in Section~\ref{sec:conclusions}.
\subsection*{Notation}

Given an integer $N$, $\NATS{N}$ denotes the integers from 1 to~$N$. Given $x\in\R$, $\lfloor x\rfloor$ denotes the greatest integer no larger than $x$ and $\lceil x\rceil$ the smallest integer no smaller than $x$. 
Given integers $k,N$, with $0\leq k\leq N$, and parameter $\varepsilon\in[0,1]$,
the Binomial cumulative distribution function is denoted as
$$\Bin(k;N,\varepsilon) \doteq \Sum{i=0}{k}\conv{N}{i}\varepsilon^i(1-\varepsilon)^{N-i}.
$$
Given the measurable function $g(x,y)$ and the probability distribution $\Pw$, we denote by $\E_\W\{g(x,y)\}$ the expected value of the random variable $g(x,y)$ and by  $\E_\W\{g(x,y)|x\}$ the expected value of $g(x,y)$ conditioned to $x$. 
The following definition is borrowed from the field of order statistics \cite{Alamo:18}.

\NEW{
\begin{definition}[Generalized Max]
Given a collection  of $N$ scalars $Q=\{ q_1, q_2, \ldots, q_N\}=\{q_i\}_{i=1}^N$, and an integer $r\in \NATS{N}$, we say that $q_r^+\in Q$ is the $r$-largest value of $Q$ if there is no more than $r-1$ elements of $Q$ strictly larger than~$q_r^+$. 
\end{definition}}
\vspace{0.2cm}

\NEW{Hence, $q^+_1$ denotes the largest value in $Q$, $q_{2}^+$ the second largest one, and so on until $q_N^+$, which is equal to the smallest one. We also use the alternative notation $q^+_r=\rmax{r}\{q_i\}_{i=1}^N$.}

\section{Uncertainty quantification using probabilistic maximization} \label{sec:prob:up:bound}

\NEW{In this section we present an initial probabilistic bound for the error $y-T(x)$ based on probabilistic maximization.
Suppose that we draw $N$  independent and identically distributed (i.i.d.) samples $\{(x_i,y_i)\}_{i=1}^N$ according to distribution $\Pw$, and we denote as}
\[
q_i\doteq|y_i-T(x_i)|, \quad i\in \NATS{N}
\]
\NEW{the absolute value of the corresponding prediction errors. A well established result \cite{Tempo:Bai:Dabbene:max:97} shows that the largest value in the sequence 
$\{q_i\}_{i=1}^N$,
i.e., $q_1^+$, 
provides a probabilistic upper bound on the random variable $q=|y-T(x)|$. Formally, given $\varepsilon\in(0,1)$ and $\delta\in(0,1)$, \cite[Theorem 1]{Tempo:Bai:Dabbene:max:97} states that if 
\begin{equation}\label{eq:bai}
N\geq \frac{1}{\varepsilon}\log(\frac{1}{\delta})    
\end{equation} then, with probability no smaller than $1-\delta$,}
$$\Pw\left\{ q>q_1^+  \right\} \leq \varepsilon.$$
\NEW{It is immediate to observe that this result provides a  first simple probabilistic scheme for uncertainty quantification:} If $N$ i.i.d. samples $\{(x_i,y_i)\}_{i=1}^N$ are drawn according to $\Pw$, with $N$ satisfying \eqref{eq:bai}, then with probability at least $1-\delta$
$$\Pw\left\{ |y-T(x)| \leq q_1^+ \right\} \geq 1-\varepsilon.$$
We notice that the required sample complexity (i.e. the number of samples $N$) depends only on $\varepsilon$ and $\delta$. Moreover, no specific assumptions are required on $T(x)$ or $\Pw$. 

However, we also note that this scheme may provide extremely conservative results, especially if the support of the random variable $q=|y-T(x)|$ is not finite and $N$ is large. In fact, suppose that $y-T(x)$ is a zero mean Gaussian random variable. Then, the probabilistic upper bound obtained from $q_1^+$ will be too conservative if one of the samples $q_i=|y_i-T(x_i)|$ departs considerably from zero, which occurs with a probability that increases with $N$. We conclude that only relying on the largest observed value of $|y-T(x)|$ hinders the computation of sharp probabilistic bounds, especially for small values of $\varepsilon$ and $\delta$, leading to a large number of samples $N$.  

In order to circumvent this issue, we resort to the following result \cite[Property 3]{Alamo:18}, which states how to obtain a probabilistic upper bound of a random scalar variable by means of the notion of generalized max (see Definition 1).

\begin{property}\label{Property:Generalized:Max}
Given $\varepsilon \in (0,1)$, $\delta\in (0,1)$ and $r\geq 1$, let $N\geq r$ be such that 
\begin{equation}\label{ineq:Bin}\Bin(r-1;N,\varepsilon)
\leq \delta.
\end{equation}
Suppose that $q\in \W\subseteq \R$ is a random scalar variable with probability distribution $\Pw$. Draw $N$ i.i.d. samples $\{q_i\}_{i=1}^{N}$ from distribution $\Pw$. Then, with a probability no smaller than $1-\delta$,
\begin{equation}\label{equ:ineq:qrplus}
\Pw\{q > \NEW{\rmax{r}}\{q_i\}_{i=1}^N\}\leq \varepsilon.
\end{equation}
\end{property}


\NEW{\begin{remark}[On Property 1]
This result is proved in \cite{Alamo:18} using techniques from the field of order statistics \cite{Ahsanullah:13}.
As discussed in  \cite{Alamo:18}, this result may be alternatively derived by applying the scenario approach with discarded constraints \cite{Campi08,CalafioreSIAM10}.
Adaptations of this result have been used in the context of chance constrained optimization \cite{alamo2019safe,mammarella2021chance}, and stochastic model predictive control \cite{Mammarella:CCTA:2020,Karg:Sergio:Lucia:Alamo:2019:probabilistic,Mammarella:IFAC:2020}. 
\end{remark}}


 Several questions arise when trying to apply Property \ref{Property:Generalized:Max} to the probabilistic error quantification problem:

\noindent
{\bf Choice of $N$}:  It was proved in \cite[Corollary 1]{Alamo:15} that the constraint $\Bin(r-1;N,\varepsilon)\leq \delta$ holds if 
\begin{equation}\label{ineq:lower:bound:on:N}
    \varepsilon N\geq r-1+\log\,\frac{1}{\delta} + \sqrt{2(r-1)\log\,\frac{1}{\delta}}.
\end{equation}
Thus, given $r$, $\delta$, and $\varepsilon$, the sample size $N$ can be obtained as the smallest integer $N$ satisfying \eqref{ineq:lower:bound:on:N}. Another possibility is to compute, by means of a numerical procedure, the smallest integer $N$ satisfying $\Bin(r-1;N,\varepsilon)\leq \delta$.

\noindent
{\bf Choice of $\delta$}: Since $1-\delta$ determines the probability of the satisfaction of the probabilistic constraint (\ref{equ:ineq:qrplus}), it is important to choose $\delta$ sufficiently close to zero. In view of \eqref{ineq:lower:bound:on:N}, we have that $N$ grows logarithmically with $\frac{1}{\delta}$. This implies that significantly small values of $\delta$ (say $\delta=10^{-6}$) can be used without an excessive impact in the number of samples $N$.

\noindent
{\bf Choice of $r$}: If $r$ is chosen to be too small, then the obtained probabilistic bounds might turn to be too conservative because the obtained upper bound would be determined by a reduced number of possible {\it extreme} values.
We notice from \eqref{ineq:lower:bound:on:N} that the larger the value of $r$, the larger the number of required samples $N$. We also derive from $\eqref{ineq:lower:bound:on:N}$ that $\frac{r-1}{N}<\varepsilon$. 
A reasonable choice for $r$ with an appropriate trade off between sample complexity $N$ and sharpness of the results is $r=\left\lfloor \frac{\varepsilon N}{2}\right\rfloor$.

\noindent
{\bf Choice of $\varepsilon$}: Parameter $\varepsilon$ determines the size of the confidence interval in the uncertainty quantification process. In uncertainty quantification, values of $\varepsilon$ much smaller than $0.05$ are not frequent.

We now state a result, which has been presented in a different context in \cite{Mammarella:CCTA:2020} and \cite{mammarella2021chance}, that shows how to obtain $N$ in such a way that \eqref{ineq:Bin} is satisfied for the particular choice $r=\left\lfloor \frac{\varepsilon N}{2}\right\rfloor$.  

\begin{lemma}\label{lemma:complexity:bound:mitad:epsilon}
Given $\varepsilon\in(0,1)$ and $\delta\in(0,1)$, suppose that  $N \ge  \frac{7.47}{\varepsilon} \ln\frac{1}{\delta}$ and $r=\left\lfloor \frac{\varepsilon N}{2}\right\rfloor$. Then $\Bin(r-1;N,\varepsilon)\leq \delta$.
\end{lemma}
\begin{proof} See the appendix of \cite{mammarella2020:CCTA:Arxiv:Version}, \NEW{where it is proved that the claim holds if $N\ge \frac{(1+\sqrt{3})^2}{\varepsilon}\ln(\frac{1}{\delta})$. The result follows from $(1+\sqrt{3})^2<7.47$.}
\end{proof}

Property \ref{Property:Generalized:Max}, along with the previous discussion on the choice of $r$, leads to Algorithm \ref{alg:probabilistic:A}, which provides a simple procedure to compute a probabilistic bound on the prediction error $y-T(x)$. 

\begin{algorithm}[H]
\caption{Probabilistic Fixed-Size Bound on  Error}
\label{alg:probabilistic:A}
\begin{algorithmic}[1]
\State \label{Alg:A:step:choice:of:N:r}
Given a predictor $T:\R^{n_x} \to \R$, and
 probability levels $\varepsilon\in(0,1)$ and $\delta\in(0,1)$, choose
\begin{equation}
N \ge  \frac{7.47}{\varepsilon} \ln\frac{1}{\delta}\quad\text{ and }\quad
r=\left\lfloor \frac{\varepsilon N}{2}\right\rfloor.
\label{N_alg}    
\end{equation}
\State Draw $N$ i.i.d. samples $\{(x_i,y_i)\}_{i=1}^N$ according to $\Pw$.
\State Compute $q_i=|y_i-T(x_i)|$, $i\in\NATS{N}$. 
\State
Return $\rho=\rmax{r} \{q_i\}_{i=1}^N$ as the probabilistic upper bound for $|y-T(x)|$.
\end{algorithmic}
\label{Algo1}
\end{algorithm}

The probabilistic guarantees of the upper bound generated with Algorithm \ref{alg:probabilistic:A} are provided in the next corollary. 
\begin{corollary}\label{corollary:A}
The output $\rho$ of Algorithm \ref{alg:probabilistic:A} satisfies, with probability no smaller than $1-\delta$, $\Pw\{|y-T(x)|> \rho\}\leq \varepsilon$.
\end{corollary}
\vspace{0,25cm}

 \proof
From Lemma \ref{lemma:complexity:bound:mitad:epsilon}, we have that the values for $N$ and $r$ obtained in step \ref{Alg:A:step:choice:of:N:r} of Algorithm \ref{alg:probabilistic:A} guarantee that $B(r-1;N,\varepsilon)\leq \delta$. 
Thus, we conclude from  Property \ref{Property:Generalized:Max} 
that $\rho=\rmax{r}\{|y_i-T(x_i)|\}_{i=1}^N$ satisfies, with probability no smaller than $1-\delta$,  
$ \Pw\{|y-T(x)|> \rho\}\leq \varepsilon$.
\QED

\subsection*{Numerical example: Algorithm~1}\label{sub:section:Algorithm:A:num:example}
Consider the function 
\begin{equation}\label{equ:fun:num:example}
y=f(x,n_1,n_2)=(10+n_1)x+10 \sin(4 x)+5+n_2.\end{equation}
We assume that $x$ is a random scalar with uniform distribution in $[-2.5,2.5]$ and $n_1$, $n_2$ are random scalars drawn from zero-mean Gaussian distributions with variances \NEW{$7$} and \NEW{$3$}, respectively. Suppose that the optimal predictor $T(x)=10x+10\sin(4x)+5$ for the random scalar $y=f(x,n_1,n_2)$ is available\footnote{We address the problem of determining predictor $T(\cdot)$ in Section \ref{sec:kernel:prediction}.}. We fix the probabilistic levels to $\varepsilon=0.05$ and $\delta=10^{-6}$, which leads to $N=2,065$ and $r=51$ (see step 1 of Algorithm \ref{alg:probabilistic:A}). We draw $N$ i.i.d. samples $\{(x_i,y_i)\}_{i=1}^N$ and obtain \NEW{$\rho=10.77$}. Thus, according to Corollary \ref{corollary:A}, with probability no smaller than $1-\delta$,  $\Pw\{|y-T(x)|> \rho\} \leq \varepsilon$. 
\NEW{We notice that for this example it is not difficult to obtain the sharpest probabilistic bounds for $y-T(x)$ corresponding to a given $x$. It suffices to notice that given $x$, $y-T(x)$ is a zero-mean Gaussian random variable with variance $7x^2+3$. Thus, using standard confidence interval analysis for a scalar Gaussian variable, we obtain that
$$\Pw\{|y-T(x)|> 1.96 \sqrt{7x^2+3}\}\leq 0.05.$$ 
\NEW{Figure \ref{f:naive_approach}} shows, for a new validation set of $N$ i.i.d. samples, the (fixed size) probabilistic bounds for $y$ provided by Algorithm 1 (i.e. $\Pw\{y\in[T(x)-\rho,T(x)+\rho]  \} \geq 1-\varepsilon$), along with the exact probabilistic bounds. We notice that Algorithm 1 fails to capture the varying size of the exact probabilistic bounds. We address this issue in the next sections.}

\begin{figure}[!h]
\centering
\includegraphics[width=1\columnwidth]{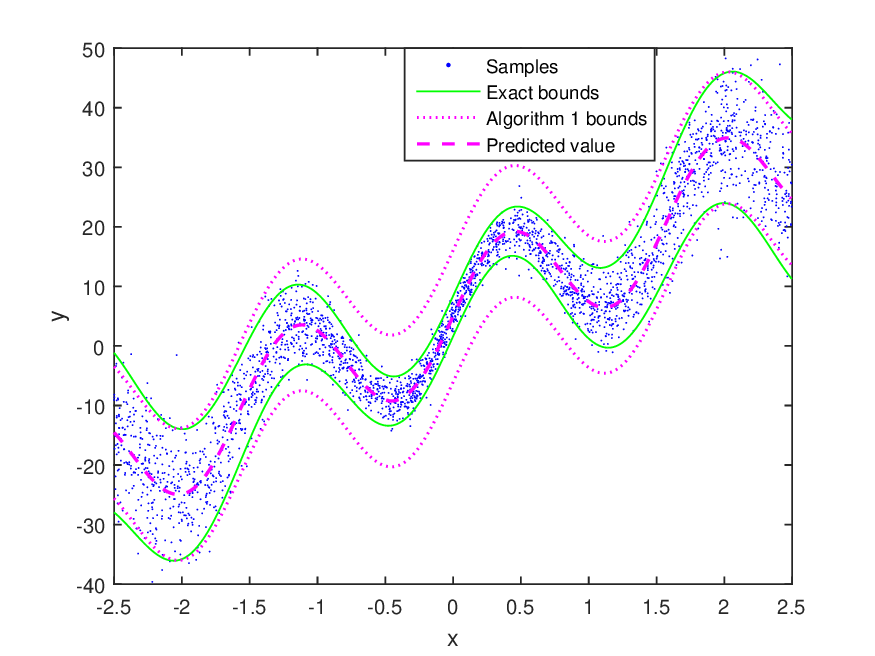}
\caption{Numerical example: \NEW{comparison between the probabilistic bounds obtained with Algorithm \ref{alg:probabilistic:A} with the exact ones. We notice that the size of the interval bounds provided by Algorithm \ref{alg:probabilistic:A} are independent of $x$}.}
\label{f:naive_approach}
\end{figure}

\section{Conditioned uncertainty quantification} \label{sec:uncertainty:quantification}

The simplicity of Algorithm \ref{alg:probabilistic:A} comes at a price: \textit{the obtained upper bound does not depend on $x$}. Clearly, this is not an issue if the error $e=y-T(x)$ is independent of $x$. However, in many situations, the expected size of the error does depend on $x$. For example, the prediction errors are often correlated with the size of the predicted variable, which in turn is correlated with $x$. From here, we infer that information on the expected error can often be obtained from~$x$.  

Under some strong assumptions, the probability distribution of $y-T(x)$ conditioned to $x$, can be computed in an explicit way. This is the case, for example, when $T(x)$ is obtained by means of Gaussian process regression \cite[\S 2]{Rasmussen:Williams:Gaussian:05} or when  exponential models are employed \cite[\S 4]{hofmann:kernel:08}. However, we notice that, although these kernel-based approaches can indeed provide estimations of the conditioned expectation $$\sigma^2(x)=\E_{\W}\{ (y-T(x))^2\,|\,x\},$$ the accuracy of the estimations will depend on the satisfaction of the underlying assumptions (i.e. Gaussian process and exponential model, respectively) and the adequate selection of the kernels (along with their hyper-parameters) used to obtain $T(x)$. There are other possibilities to obtain conditioned error quantification, like sensitivity analysis, techniques based on Fisher information matrix, bootstrapping, etc. \cite{VILLAVERDE201945}, \cite{mcclarren2018:Uncertainty:Quantification}.

We also mention here Parzen method \cite{parzen1962estimation}, which serves to estimate the probability density function of a random variable. More general multivariate kernel-based generalizations are also available (see e.g., \cite{pelletier2005kernel:density:estimation}). In these methods, an estimation $\hat{\sigma}(x)$ of $\sigma(x)$ is obtained from 
\begin{equation}\label{equ:Parzen}
    \NEW{\hat{\sigma}^2(x)= \frac{\Sum{i=1}{M} (y_i-T(x_i))^2\Gamma(x,x_i)}{\Sum{i=1}{M} \Gamma(x,x_i)}}, 
\end{equation} 
where $\Gamma:\R^{n_x} \times \R^{n_x} \to \R$ is an appropriately chosen function and $\{(x_i,y_i)\}_{i=1}^{M}$ are i.i.d. samples drawn from $\Pw$. Under non very restrictive constraints \cite{parzen1962estimation,pelletier2005kernel:density:estimation}, the provided estimation $\hat{\sigma}(x)$ converges to the actual value $\sigma(x)$ as $M$ tends to infinity. 

For a fixed value of $x$, $d=(y-T(x))^2$ is a random non-negative variable with expectation $\sigma^2(x)$. Thus, we can resort to the Markov inequality \cite{Papoulis:2002, huber2019halving:Markov:Chebyshev:Chernoff}, to obtain
$$ \Pw\{d \geq \xi \sigma^2(x)\, \}\leq \frac{1}{\xi}, \;\forall \xi >0.$$
Thus, choosing $\xi=\fracg{1}{\varepsilon}$, we obtain $ \Pw\left\{d \geq \frac{\sigma^2(x)}{\varepsilon}\right\}\leq~\varepsilon$. Equivalently,
$$ \Pw\left\{|y-T(x)| \geq \frac{\sigma(x)}{\sqrt{\varepsilon}} \right\}\leq \varepsilon.$$
The obtained probabilistic upper bound suffers from the following two limitations: (i) generally, $\sigma(x)$ is unknown and only a rough estimation is available (as the ones commented before), and (ii) Markov inequality yields overly conservative results in many situations \cite{huber2019halving:Markov:Chebyshev:Chernoff}. A meaningful exception to this is when the errors $y-T(x)$ are of Gaussian nature. In this case, using the Chi-squared distribution \cite{Papoulis:2002}, sharp probabilistic bounds of the form $\Pw\{|y-T(x)|\geq \gamma_\varepsilon \sigma(x)\}=\varepsilon$, 
can be obtained. 

In order to avoid these limitations, we can again resort to probabilistic maximization. 
Suppose that an estimation $\hat{\sigma}(x)$ of $\sigma(x)$ is available. Suppose also that $\hat{\sigma}(x)>0$, for all $x\in \R^{n_x}$. We could define the scaling factor $\gamma$ as
$$ \gamma=\frac{|y-T(x)|}{\hat{\sigma}(x)}.$$
With this definition, any probabilistic upper bound $\bar{\gamma}$ on $\gamma$ would provide a probabilistic upper bound on $|y-T(x)|$. 
That is, 
$$\Pw\{\gamma > \bar{\gamma} \}\leq \varepsilon \; \Rightarrow \Pw\{|y-T(x)| > \bar{\gamma} \hat{\sigma}(x) \}\leq \varepsilon.$$
This means that we could slightly modify Algorithm \ref{alg:probabilistic:A} to obtain a novel algorithm capable of obtaining a probabilistic upper bound conditioned by the value of $x$. This idea is implemented in Algorithm \ref{alg:probabilistic:B}.

\begin{algorithm}[H]
\caption{Conditioned Probabilistic Bound on Error}
\label{alg:probabilistic:B}
\begin{algorithmic}[1]
\State \label{Alg:B:step:choice:of:N:r}
Given a predictor $T:\R^{n_x} \to \R$, 
an estimator $\NEW{\hat{\sigma}}:\R^{n_x} \to (0,\infty)$, of $\NEW{\sqrt{\E_\W\{(y-T(x))^2\,|\,x\}}}$,  probability levels $\varepsilon\in(0,1)$ and $\delta\in(0,1)$, choose 
$$
N \ge  \frac{7.47}{\varepsilon} \ln\frac{1}{\delta}\quad\text{ and }\quad
r=\left\lfloor \frac{\varepsilon N}{2}\right\rfloor.
$$\State Draw $N$ i.i.d. samples $\{(x_i,y_i)\}_{i=1}^N$ according to $\Pw$.
\State Compute $\gamma_i=\frac{|y_i-T(x_i)|}{\hat{\sigma}(x_i)}$, $i\in\NATS{N}$. 
\State
Return $\bar{\gamma}=\NEW{\rmax{r}}\{\gamma_i\}_{i=1}^N$, as probabilistic upper bound for $\frac{|y-T(x)|}{\hat{\sigma}(x)}$.
\end{algorithmic}
\end{algorithm}

The following Corollary states the probabilistic guarantees of the output $\bar{\gamma}$ of Algorithm \ref{alg:probabilistic:B}.
\begin{corollary} \label{corollary:B}
The output $\bar{\gamma}$ of Algorithm \ref{alg:probabilistic:B} satisfies, with probability no smaller than $1-\delta$,
$$ \Pw\{|y-T(x)|> \bar{\gamma}\hat{\sigma}(x)\}\leq \varepsilon.$$
\end{corollary}
\vspace{0,25cm}

\proof The proof follows the same lines as the proof of Corollary \ref{corollary:A}. That is, we infer from Property \ref{Property:Generalized:Max} and Lemma \ref{lemma:complexity:bound:mitad:epsilon} that the proposed choice of $N$ and $r$ guarantees that, with probability no smaller than $1-\delta$,  
$$ \Pw\left\{\gamma = \frac{|y-T(x)|}{\hat{\sigma}(x)} > \bar{\gamma}\right\} \leq \varepsilon.$$
Thus, we conclude $ \Pw\{|y-T(x)|> \bar{\gamma}\hat{\sigma}(x)\} \leq \varepsilon$. \QED
\vspace{0,25cm}

\begin{remark}[On normalization of $\hat{\sigma}(x)$]\label{remark:normalization}
We notice that the upper bound obtained by means of Algorithm \ref{alg:probabilistic:B} provides identical results when the estimator $\hat{\sigma}(x)$ is replaced by a scaled version $\hat{\sigma}_\xi(x)=\xi\hat{\sigma}(x)$, where $\xi>0$. Thus, multiplicative errors in the estimation of $\sigma(x)$ are corrected in an implicit way by the algorithm.   
\end{remark}

\NEW{\begin{remark}[Difference with convex scenario approaches]
Scenario approaches (see e.g. \cite{campi2009interval,garatti2019class}) obtain both the estimator and probabilistic guarantees in a single optimization problem that requires a number of samples that increases both with the dimension of the regressor used in the predictive model and the number of samples that are allowed to violate the interval predictions. Our approach can be applied to any given predictor $T(\cdot)$ and has a sample complexity that does not depend on the dimension of the regressor. This allows us to consider kernel approaches in a possible infinite dimensional lifted space (see Section \ref{sec:kernel:prediction}).
\end{remark}}

\section{Uncertainty quantification for finite families of estimators} \label{sec:finite:families}

The probabilistic bounds proposed for error $e=y-T(x)$ depend not only on the intrinsic random relationship between $x$ and $y$ (joint probability distribution), but also on the choice of the estimators $T(\cdot)$ and $\hat{\sigma}(\cdot)$. Since there exists a myriad of possibilities for choosing $T(\cdot)$ and $\hat{\sigma}(\cdot)$, we now analyze the problem of choosing among a finite family $\F$ of possible pairs $(T(\cdot),\hat{\sigma}(\cdot))$ the one that minimizes the size of the obtained probabilistic bounds. The following result states the relationship between the cardinality $n_\F$ of $\F$, and the probabilistic specifications $(\varepsilon, \delta)$, with the number of samples required to obtain the corresponding bounds. 

\begin{theorem}
\label{theorem:families}
Consider the finite family of candidate estimators
$$ \F=\set{(T_j(\cdot), \hat{\sigma}_j(\cdot))}{j\in \NATS{n_\F}},$$
where $T_j:\R^{n_x}\to \R$ and $\hat{\sigma}_j:\R^{n_x} \to (0,\infty)$ for every $j\in\NATS{n_\F}$.
Given $\varepsilon \in (0,1)$, $\delta\in (0,1)$ and $r\geq 1$, let $N\geq r$ be such that $\Bin(r-1;N,\varepsilon)\leq \frac{\delta}{n_\F}$.
Draw $N$ i.i.d. samples $\{(x_i,y_i)\}_{i=1}^{N}$ from distribution $\Pw$ and denote \begin{equation}\label{equ:def:bar:gamma:j}
\bar{\gamma}_j \doteq \NEW{\rmax{r}}\left\{\fracg{|y_i-T_j(x_i)|}{\hat{\sigma}_j(x_i)}\right\}_{i=1}^{N},\; j\in\NATS{n_\F}.
\end{equation}
Then, with a probability no smaller than $1-\delta$,
$$ \Pw\{|y-T_j(x)| > \bar{\gamma}_j \hat{\sigma}_j(x)\}\leq \varepsilon, \; j\in \NATS{n_\F}.$$ 
\end{theorem}
\vspace{0.25cm}

\proof Denote $\delta_\F$ the probability that at least one of the randomly obtained scalars  $\{\bar{\gamma}_j\}_{j=1}^{n_\F}$, obtained from the random multi-sample $\{(x_i,y_i)\}_{i=1}^N$, does not satisfy the constraint 
\begin{equation} \label{ineq:Ej:varpesilon} E_j(\bar{\gamma}_j) \doteq \Pw\left\{\fracg{|y-T_j(x)|}{\hat{\sigma}_j(x)} > \bar{\gamma}_j \right\} \leq \varepsilon.
\end{equation}
Thus, 
\begin{eqnarray*}
\delta_\F &=& \PwN \{ \varepsilon < \max\limits_{j\in \NATS{n_\F}} E_j (\bar{\gamma}_j) \} 
 \\
& \leq &  \Sum{j=1}{n_\F} \PwN \{ \varepsilon < E_j (\bar{\gamma}_j) \}  \leq \Sum{j=1}{n_\F}\frac{\delta}{n_\F} = \delta.
\end{eqnarray*}
We notice that the last inequality is due to the assumption $B(r-1;N,\varepsilon)\leq \frac{\delta}{n_\F}$,  \eqref{equ:def:bar:gamma:j} and Property \ref{Property:Generalized:Max}. 
Thus, with probability no smaller than $1-\delta_\F \geq 1-\delta$, inequality \eqref{ineq:Ej:varpesilon} is satisfied for every $j\in \NATS{n_\F}$. \hfill \QED

\NEW{\begin{remark}[Sample complexity for finite families]
In view of Lemma \ref{lemma:complexity:bound:mitad:epsilon}, it suffices to draw $N=\left\lceil \frac{7.47}{\varepsilon} \log\frac{n_\F}{\delta}\right\rceil$ i.i.d. samples from $\Pw$ to obtain a probabilistic uncertainty quantification for the complete finite family $\F$. In order to select the best pair  $(T_j(\cdot),\hat{\sigma}_j(\cdot))$ in $\F$, one could choose the index $j\in\NATS{n_\F}$ providing the sharpest probabilistic uncertainty bounds. That is, the one minimizing $\Sum{i=1}{N} \bar{\gamma}_j \hat{\sigma}_j(x_i)$. Since $n_\F$ enters in a logarithmic way in the sample complexity bound, large values for $n_\F$ are affordable. In this case, the search for the most appropriate pair $(T_j(\cdot),\hat{\sigma}_j(\cdot))$ does not need to be exhaustive, and sub-optimal search in the finite family $\F$ could be envisaged (since the probabilistic bounds provided are valid for every member of the family $\F$).
\end{remark}}



\section{Kernel Central Prediction and Uncertainty Quantification} \label{sec:kernel:prediction}

Suppose that $M$ i.i.d. samples $\{(x_i,y_i)\}_{i=1}^M$ are available. We now address the design of the predictor $T(x)$ by means of kernel methods while guaranteeing that the procedure also provides us with an estimation of $\sigma(x)$. Given $x$, let us define the loss functional
\begin{equation}\label{equ:def:J:theta}
    J(\theta;x)= \theta^T \Sigma_\theta\theta + \Sum{i=1}{M}(y_i-\theta\T\varphi(x_i))^2\Gamma(x,x_i),
\end{equation}
where $\varphi:\R^{n_x}\to \R^{n_\theta}$ is the regressor function and $\theta\T \Sigma_\theta \theta$ is a regularization term. A possible choice is $\Sigma_\theta=\tau \Id$, where $\tau>0$. Finally, $\Gamma:\R^{n_x} \times \R^{n_x} \to \R$ is an appropriately chosen weighting function. We assume that  $\Gamma(x,z)$ is a decreasing function of $\|x-z\|$, where $\|\cdot\|$ is a given norm. For example, $\Gamma(x,z)=\exp(-\lambda \|x-z\|)$,   where $\lambda>0$. 

As it is usual in machine learning, \NEW{for given $x$}, a {\it central} estimation for $y$ is provided by $T(x)=\theta_c\T(x)\varphi(x)$, where $\theta_c(x)$ is given by $\theta_c(x) = \arg\min\limits_\theta J(\theta;x)$.
We notice that the proposed estimator is a weighted least square estimator with a ridge regression regularization term \cite{hoerl1970ridge,hastie2009elements}. 

There exists two possibilities to obtain predictor $T(x)$ and local estimations $\hat{y}_i(x)=\theta_c\T(x)\varphi(x_i)$, $i\in\NATS{M}$ (which will be needed to compute the Parzen estimator for $\sigma(x)$):\\
\vspace{0.2cm}
\noindent

{\bf Based on $\varphi(\cdot)$}: Since $J(\theta;x)$ is a strictly convex quadratic function of $\theta$, the optimal value $\theta_c(x)$ can be obtained determining the value of $\theta$ for which the gradient of $J(\theta;x)$ with respect to $\theta$ vanishes.\\
\vspace{0.2cm}

\noindent
 {\bf Based on a kernel formulation}: Defining the kernel function $K(\cdot,\cdot)$ as $K(x_a,x_b)=\varphi\T (x_a)\Sigma_\theta^{-1} \varphi\T(x_b)$, \NEW{the estimation $T(x)$, along with the local estimations $\hat{y}_i(x)$, $i\in\NATS{M}$, can be obtained in an explicit way by means of well-known kernel tricks (see e.g., \cite[\S 14.4.3]{murphy2012machine,vovk2013kernel} and references therein)}. In this case, the kernel formulation allows to approach the regression problem in a possibly infinite dimensional lifted space \cite{hastie2009elements}.

Once the local estimations $\{\hat{y}_i(x)\}$, $i\in\NATS{M}$ have been computed, the estimation for $\sigma(x)$ can be obtained from the following local Parzen estimator:
\begin{equation} \label{eq:Parzen:Kernel}\NEW{\hat{\sigma}^2(x)= \frac{\Sum{i=1}{M} (y_i-\hat{y}_i(x))^2\Gamma(x,x_i)}{\Sum{i=1}{M} \Gamma(x,x_i)}}. \end{equation}
See the Appendix for a detailed description on how to obtain predictors $T(x)$ and $\hat{\sigma}(x)$ for both considered possibilities (i.e. based on regressor $\varphi(\cdot)$ or based on a kernel formulation).

\NEW{As commented before, the Parzen estimator converges, under non very restrictive assumptions, to the actual value $\sigma(x)$ as $M$ tends to infinity (\cite{parzen1962estimation,pelletier2005kernel:density:estimation}). A too reduced number of samples $M$, or a non appropriate choice for weighting factors $\Gamma(\cdot,\cdot)$, may translate into a degraded estimation of $\sigma(x)$, which will not affect the probabilistic properties of the obtained bounds (that are guaranteed by Theorem \ref{theorem:families}), but will lead to more conservative bounds.} We also notice that an additional set of $N$ i.i.d samples is required to compute the scaling factor $\bar{\gamma}$ in Algorithm \ref{alg:probabilistic:B}.  

\section{Numerical example: Kernel finite families}
We revisit now the numerical example proposed in Section \ref{sec:prob:up:bound}. For the predictor $T(\cdot)$ we consider a radial basis function kernel $k(x_a,x_b)=50\NEW{\exp(-\frac{|x_a-x_b|^2}{0.2})}$, and for the estimator $\hat{\sigma}(x)$ the Parzen estimator in \eqref{eq:Parzen:Kernel}, 
where $M=2,065$ and the pairs $\{(\tilde{y}_i,\tilde{x}_i)\}_{i=1}^M$ are i.i.d. samples from $\Pw$.  We consider a family of weighting functions $\Gamma(x,z)=\NEW{\exp(-\lambda |x-z|)}$, where \NEW{$\lambda\in\NATS{10}$}. Thus, the finite family $\F$ consists of each of the $n_\F=10$ possible pairs $(T_j(\cdot),\hat{\sigma}_j(\cdot))$ that can be obtained with the  $n_\F$ values considered for the hyper-parameter \NEW{$\lambda$} using the methodology proposed in Section \ref{sec:kernel:prediction}. Setting $\varepsilon=0.05$, $\delta=10^{-6}$ and $n_\F=10$, we obtain from \NEW{Theorem}~\ref{theorem:families} and Lemma \ref{lemma:complexity:bound:mitad:epsilon} that the choice $N=2407$ and $r=60$ is sufficient to obtain a probabilistic uncertainty quantification valid for all the members of the family. The value of \NEW{$\lambda$} minimizing the size of the obtained probabilistic bounds is attained at \NEW{$\lambda=1$}. The resulting scaling parameter is \NEW{$\bar{\gamma}=2.15$}. \NEW{See Figure \ref{f:kernel:families} for a comparison of the results obtained for the same validation set that was used to generate Figure \ref{Algo1}. The ratio of violation in the validation set for the proposed finite family approach was 0.0332, whereas it was 0.0511 for the exact probabilistic bounds.}

\begin{figure}[t]
\centering
\includegraphics[trim= .8cm 0cm 1.2cm 0.5cm, clip=true,width=1\columnwidth]{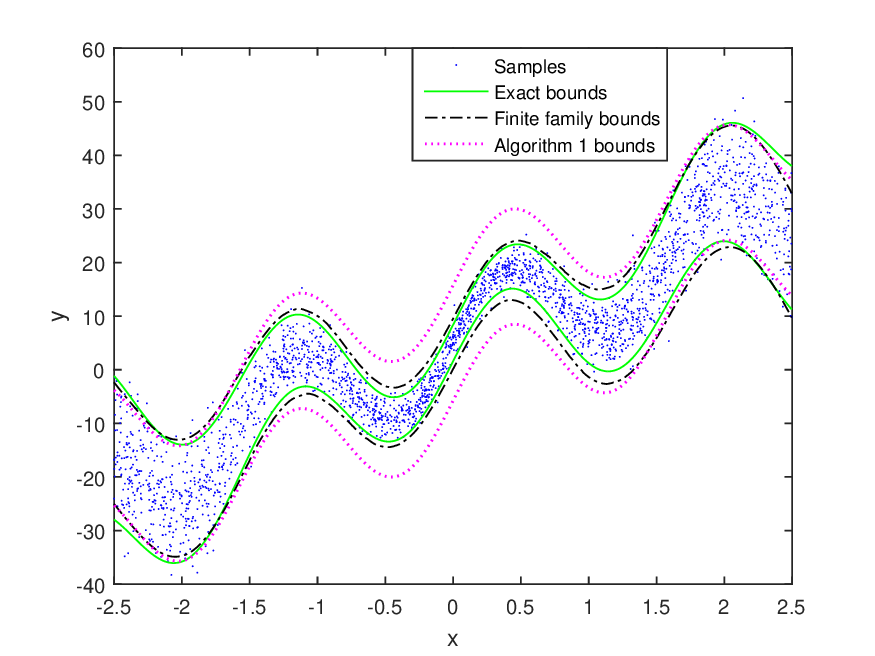}
\caption{Probabilistic upper bounds obtained by means of a finite family of kernel estimators, Algorithm 1 bounds and exact bounds. We notice that the new probabilistic bounds are modulated by the estimated value $\hat{\sigma}(x)$.}
\label{f:kernel:families}
\end{figure}


\section{Conclusions}  
\label{sec:conclusions}
In this paper, we proposed a methodology to obtain a probabilistic upper bound on the absolute value of the prediction error via a sample-based approach.  
We provided a series of approaches of increasing complexity. All the proposed techniques share the desirable characteristic of requiring a number of observations which is independent of the prediction model complexity.
This is made possible by the exploitation of a probabilistic scaling scheme.

\section*{Fundings}  
This work was supported  by the Agencia Estatal de Investigaci\`on (AEI)-Spain under Grant PID2019-106212RB-C41/AEI/10.13039/501100011033 and partially funded by the Italian IIT and MIUR within the 2017 PRIN (N. 2017S559BB).

\bibliographystyle{plain}
\bibliography{BIBLIO.bib}

\section*{Appendix: Computation of estimators $T(x)$ and $\hat{\sigma}(x)$}

Given a specific test point $x$, and a given regressor function
$\varphi(\cdot)$, the local ridge regression is $T(x)=\varphi(x)\T \theta_c(x)$, where $\theta_c(x)$ is the minimizer of \eqref{equ:def:J:theta}. 
That is, 
$$ \theta_c(x) = \arg\min\limits_\theta\,\theta^T \Sigma_\theta\theta + \Sum{i=1}{M}(y_i-\theta\T\varphi(x_i))^2\Gamma(x,x_i).$$

We first notice that in some local regression approaches, $\Gamma(x,x_i)$ is set to zero if $x_i$ does not belong to a neighbourhood of $x$. Similarly, function $\Gamma(x,\cdot)$ could be tuned in such a way that only $L\leq M$ samples  $x_i$ satisfy $\Gamma(x,x_i)>0$ (e.g. those  closest to $x$). This sort of strategies are specially relevant in kernel methodologies because, as it will be shown later in this appendix, their implementation requires the solution of a system of equations in $L$ variables. Hence, the complexity of kernel approaches can be kept to affordable levels by adjusting the design parameter $L$.

Since $\Gamma(x,x_i)=0$ implies that the pair $(x_i,y_i)$ has no effect on the value of $\theta_c(x)$, we will consider only the pairs $(x_i,y_i)$ for which $\Gamma(x,x_i)>0$. We denote such pairs as $\{\tx_j,\ty_j\}_{j=1}^L$, where $1\leq L\leq M$. Thus,
$$ \theta_c(x) = \arg\min\limits_\theta\;\theta^T \Sigma_\theta\theta + \Sum{j=1}{L}(\ty_j-\theta\T\varphi(\tx_j))^2\Gamma(x,\tx_j).$$
For notational convenience, we denote 
\begin{eqnarray*} \varphi_j & \doteq & \varphi(\tx_j),\; j=1,\ldots,L,\\
\Gamma_j & \doteq & \Gamma(x,\tx_j),\; j=1,\ldots,L.
\end{eqnarray*}
We first address the case in which regressor function $\varphi(\cdot)$ is available. Later, we address the situation in which the estimators are obtained in terms of a kernel formulation, i.e., when the estimators are not directly expressed in terms of a regressor function, but of a kernel function.

From $(\ty_j-\varphi_j\T \theta)^2=\theta\T \varphi_j\varphi_j\T \theta-2\ty_j(\theta\T\varphi_j)+\ty_j^2$, 
we obtain that $\theta_c(x)$ is the minimizer of 
$$ \theta\T \left( \Sigma_\theta + \Sum{j=1}{L} \Gamma_j\varphi_j\varphi_j\T \right)\theta- 2\theta\T \left( \Sum{j=1}{L}\Gamma_j\ty_j\varphi_j\right). $$
That is, 
\begin{eqnarray}
\theta_c(x) &=& \left( \Sigma_\theta + \Sum{j=1}{L} \Gamma_j\varphi_j\varphi_j\T \right)^{-1} \left( \Sum{j=1}{L}\Gamma_j\ty_j\varphi_j\right)\nonumber \\
&= &(\Sigma_\theta+ R D R\T )^{-1} R D\vv{y}_\D, \label{equ:theta:noKernel}
\end{eqnarray}
where $R\in \R^{n_\theta\times L}$, $D\in \R^{L\times L}$ and $\vv{y}_\D\in \R^L$ are given by  
\begin{eqnarray*}
R&=&\bmat{cccc} \varphi_1  & \varphi_2 & \ldots & \varphi_L\emat,\\
\vv{\ty}_\D &=& \bmat{cccc} \ty_1 & \ty_2 & \ldots & \ty_L\emat\T,\\ 
D&=&\diag \left(\Gamma_1, \Gamma_2, \ldots, \Gamma_L \right)>0.
\end{eqnarray*}

Thus, the estimator of $y$, given $x$, is  
\begin{eqnarray*}T(x)&=& \varphi(x)\T \theta_c(x)\\ &=& 
\varphi(x)\T (\Sigma_\theta+ R D R\T )^{-1} R D\vv{y}_\D.
\end{eqnarray*}
We also define the local errors 
$$
\tilde{e}_j(x)=\ty_j - \varphi_j\T \theta_c(x) , \;j=1,\ldots, L.
$$

The resulting estimator $\hat{\sigma}(x)$ (see equation \ref{eq:Parzen:Kernel}) is

$$\hat{\sigma}^2(x)  =  \frac{\Sum{j=1}{L}\Gamma_j\tilde{e}_j^2(x) }{\Sum{j=1}{L} \Gamma_j}. $$
Given $\vv{e}\in \R^L$, we denote $\|\vv{e}\|_D=\sqrt{\vv{e}\T D \vv{e}}$. With this notation we obtain
\begin{eqnarray} \hat{\sigma}(x) &=& \frac{1}{\sqrt{\trace{D}}} \left \| \bv  \tilde{e}_1(x) \\ \tilde{e}_2(x) \\
\vdots \\ \tilde{e}_L(x)\ev \right\|_D \label{equ:ref:e:D}\\
& = & \frac{1}{\sqrt{\trace D}}\left\| \bv  \ty_1 - \varphi_1\T \theta_c(x)\\ \ty_2-\varphi_2\T \theta_c(x) \\ 
\vdots \\ \ty_L-\varphi_L\T \theta_c(x)\ev \right\|_D \nonumber\\
& = & \frac{1}{\sqrt{\trace D}}\left\| (\Id- R\T (\Sigma_\theta+ R D R\T )^{-1} R D)\vv{y}_\D\right\|_D.\nonumber\end{eqnarray}

Since the weighting factors $\Gamma(x,\cdot)$ depend on $x$, $\theta_c(x)$ depends on $x$. Thus, the proposed procedure has to be repeated each time estimators $T(\bar{x})$ and $\hat{\sigma}(\bar{x})$ are required for a particular test point $\bar{x}$.  

We recall now the following well known matrix equality \cite[Subsection 1.3]{InversionLemma:Henderson:81}, \cite[Corollary 4.3.1]{murphy2012machine}:
$$(H - RF^{-1} R\T )^{-1} RF^{-1}  =H^{-1} R(F-R\T H^{-1}R)^{-1},$$
which is valid whenever $H$ and $F$ are non singular matrices.
In view of this equality, we obtain from \eqref{equ:theta:noKernel} the following expression for $\theta_c(x)$
\begin{eqnarray*}	
	\theta_c(x) &= &(\Sigma_\theta+ R D R\T )^{-1} R D\vv{y}_\D\\
	&= &-((-\Sigma_\theta)- R D R\T )^{-1} R D\vv{y}_\D\\
	&= & -(-\Sigma_\theta)^{-1} R(D^{-1}- R\T (-\Sigma_\theta)^{-1}R)^{-1}\vv{y}_\D\\ 
	&=&\Sigma_\theta^{-1} R(D^{-1}+ R\T \Sigma_\theta^{-1}R)^{-1}\vv{y}_\D.
\end{eqnarray*}
Thus, given $x$ and $\varphi_x=\varphi(x)$, we obtain the following estimation $T(x)$ for $y$
\begin{eqnarray*}
T(x) &=& \varphi_x\T\theta_c(x) \\
&=& 	\varphi_x\T 
\Sigma_\theta^{-1} R(D^{-1}+ R\T \Sigma_\theta^{-1}R)^{-1}\vv{y}_\D \\
& = &\bv \varphi_x\T \Sigma_\theta^{-1} \varphi_1 \\ \varphi_x\T \Sigma_\theta^{-1}\varphi_2\\  \vdots\\ \varphi_x\T \Sigma_\theta^{-1} \varphi_L \ev\T \left(D^{-1}+\mathbb{K} \right)^{-1} \vv{y}_\D, 
\end{eqnarray*}
where 
\begin{eqnarray*}
\mathbb{K} &=& R\T \Sigma_\theta^{-1} R \\
&=& \bv \varphi_1\T\\ \varphi_2\T \\ \vdots\\ \varphi_L\T\ev \Sigma_\theta^{-1}\bmat{cccc} \varphi_1& \varphi_2 &\ldots& \varphi_L\emat \\
&= & \bmat{cccc} \varphi_1 \T \Sigma_\theta^{-1}\varphi_1 & \varphi_1 \T \Sigma_\theta^{-1}\varphi_2 & \ldots & \varphi_1 \T \Sigma_\theta^{-1}\varphi_L \\
	\varphi_2 \T \Sigma_\theta^{-1}\varphi_1 & \varphi_2 \T \Sigma_\theta^{-1}\varphi_2 & \ldots & \varphi_2 \T \Sigma_\theta^{-1}\varphi_L \\
		\vdots & \vdots & \ddots & \vdots \\
	\varphi_L \T \Sigma_\theta^{-1}\varphi_1 & \varphi_L \T \Sigma_\theta^{-1}\varphi_2 & \ldots & \varphi_L \T \Sigma_\theta^{-1}\varphi_L  \emat.
\end{eqnarray*}
If we now define the kernel function $K(\cdot,\cdot)$ as 
$$K(x_a,x_b)=\varphi\T (x_a)\Sigma_\theta^{-1} \varphi\T(x_b),$$ we obtain
\begin{equation}\label{equ:kernel:estimatior:s}
    T(x) = \bv K(x,\tx_1)\\ K(x,\tx_2)\\  \vdots\\ K(x,\tx_L)  \ev\T \left(D+\mathbb{K} \right)^{-1} \vv{y}_\D, 
\end{equation} 
where 
$$ \mathbb{K}=\bmat{cccc} K(\tx_1,\tx_1) & K(\tx_1,\tx_2) & \ldots & K(\tx_1,\tx_L) \\
		K(\tx_2,\tx_1) & K(\tx_2,\tx_2) & \ldots & K(\tx_2,\tx_L) \\
		\vdots & \vdots & \ddots & \vdots \\
		K(\tx_L,\tx_1) & K(\tx_L,\tx_2) & \ldots & K(\tx_L,\tx_L) \emat.
$$
The kernel function must satisfy the Meyer's condition, i.e. matrix $\mathbb{K}$ should be semidefinite positive for any collection of $L$ points \cite{scholkopf2002learningwithKernels}. Popular kernel functions satisfying this condition are:

	\begin{itemize}
		\item Linear: $K(x_a,x_b) =x_a\T x_b $.
		\item Polynomial: $K(x_a,x_b) = (c+x_a\T x_b )^d$.
		\item Radial  : $K(x_a,x_b) = \exp\left(\frac{-\|x_a-x_b\|^2}{d^2}\right)$.
		\item Sigmoidal:  $K(x_a,x_b)=\tanh (c_a x_a\T x_b + c_b)$.
	\end{itemize}
The local errors $\{\tilde{e}_j\}_{j=1}^L$ are obtained from
\begin{eqnarray*}
\tilde{e}_j(x) &=& \ty_j-\varphi_j\T \theta_c(x) \\
& = & \ty_j - 	\varphi_j\T 
\Sigma_\theta^{-1} R(D^{-1}+ R\T \Sigma_\theta^{-1}R)^{-1}\vv{y}_\D \\
& = & \ty_j- \bv \varphi_j\T \Sigma_\theta^{-1} \varphi_1 \\ \varphi_j\T \Sigma_\theta^{-1}\varphi_2\\  \vdots\\ \varphi_j\T \Sigma_\theta^{-1} \varphi_L \ev\T \left(D^{-1}+\mathbb{K} \right)^{-1} \vv{y}_\D, \\ &= &
\ty_j- \bv K(\tx_j, \tx_1 )\\ K(\tx_j, \tx_2 )\\  \vdots\\ K(\tx_j, \tx_L ) \ev\T \left(D^{-1}+\mathbb{K} \right)^{-1} \vv{y}_\D.
\end{eqnarray*}
Thus,  we obtain
\begin{eqnarray*}
\bv  \tilde{e}_1(x) \\ \tilde{e}_2(x) \\
\vdots \\ \tilde{e}_L(x)\ev &=&\left( \Id-\mathbb{K} \left(D+\mathbb{K} \right)^{-1} \right)\vv{y}_\D.   \end{eqnarray*}
We finally conclude from equation \eqref{equ:ref:e:D} that
$$ \hat{\sigma}(x) = \frac{1}{\sqrt{\trace D}} \left\| \left( \Id-\mathbb{K} \left(D+\mathbb{K} \right)^{-1} \right)\vv{y}_\D\right\|_D.  $$

\end{document}